\documentclass{amsart}

\usepackage{amsmath}
\usepackage{amssymb}
\usepackage{amsthm}

\newtheorem{thm}{Theorem}
\newtheorem{cor}[thm]{Corollary}
\newtheorem{Lem}[thm]{Lemma}
\newtheorem{rmk}[thm]{Remark}

\DeclareMathOperator{\PIexp}{exp}
\DeclareMathOperator{\var}{var}
\DeclareMathOperator{\Id}{Id}
\DeclareMathOperator{\Der}{Der}
\DeclareMathOperator{\Span}{span}
\DeclareMathOperator{\Supp}{Supp}
\DeclareMathOperator{\Char}{char}
\DeclareMathOperator{\ad}{ad}

\title[Growth of differential identities]{Growth of differential identities}

\author[C. Rizzo]{Carla Rizzo}\address{Dipartimento di Matematica e
Informatica
\\Universit\`a di Palermo \\ Via Archirafi 34, 90123 Palermo, Italy}
\email{carla.rizzo@unipa.it}



\begin{document}

\begin{abstract}

In this paper we study the growth of the differential identities of some algebras with derivations, i.e., associative algebras where a Lie algebra $L$ (and its universal enveloping algebra $U(L)$) acts on them by derivations. In particular, we study in detail the differential identities and the cocharacter sequences of some algebras whose sequence of differential codimensions has polynomial growth. Moreover, we shall give a complete description of the differential identities of the algebra $UT_2$ of $2\times 2$ upper triangular matrices endowed with all possible action of a Lie algebra by derivations. Finally, we present the structure of the differential identities of the infinite dimensional Grassmann $G$ with respect to the action of a finite dimensional Lie algebra $L$ of inner derivations.

\end{abstract}

\maketitle

\section{Introduction}
Let $A$ be an associative algebra over a field $F$ of characteristic zero and assume that a Lie algebra $L$ acts on it by derivations. Such an action can be naturally extended to the action of the universal enveloping algebra $U(L)$ of $L$ and in this case we say that $A$ is an algebra with derivations or an $L$-algebra. In this context it is natural to define the differential identities of $A$, i.e., the polynomials in non-commutative variables $x^{h}=h(x)$, $h\in U(L)$, vanishing in $A$.



An effective way of measuring the differential identities satisfied by a given $L$-algebra $A$ is provided by its sequence of differential codimensions $c_{n}^{L}(A)$, $n=1,2,\dots$. The $n$th term of such sequence measures the dimension of the space of multilinear differential polynomials in $n$ variables of the relatively free algebra with derivations of countable rank of $A$. Since in characteristic zero, by the multilinearization process, every differential identity is equivalent to a system of multilinear ones, the sequence of differential codimensions of $A$ gives a quantitative measure of the differential identities satisfied by the given $L$-algebra. Maybe the most important feature of this sequence proved by Gordienko in \cite{Gordienko2013JA} is that in case $A$ is a finite dimensional $L$-algebra, $c_{n}^{L}(A)$ is exponentially bounded. Moreover, he determined the exponential rate of growth of the sequence of differential codimension, i.e, he proved that for any finite dimensional $L$-algebra $A$, the limit $\lim_{n\to \infty}\sqrt[n]{c_{n}^{L}(A)}$ exists and is a non-negative integer. Such integer, denoted $\PIexp^L(A)$, is called the differential PI-exponent of the algebra $A$ and it provides a scale allowing us to measure the rate of growth of any finite dimensinal $L$-algebra. As a consequence of this result it follows that the differential codimensions of a finite dimensional $L$-algebra $A$ are either polynomially bounded or grow exponentially. Hence no intermediate growth is allowed.


When studying the polynomial identities of an $L$-algebra $A$, one is lead to consider $\var ^{L}(A)$, the $L$-variety of algebras with
derivations generated by $A$, that is the class of $L$-algebras satisfying all differential identities satisfied by $A$. Thus we define the growth of $\mathcal{V}= \var ^{L}(A)$ to be the growth of the sequence $c_{n}^{L}(\mathcal{V})=c_{n}^{L}(A)$, $n=1,2,\dots$ and we say that a variety $\mathcal{V}$ has almost polynomial growth if $\mathcal{V}$ has exponential growth but every proper subvariety has polynomial growth. Since the ordinary polynomial identities and corresponding codimensions are obtained by let $L$ acting on $A$ trivially (or $L$ is the trivial Lie algebra), the algebra $UT_2$ of $2 \times 2$ upper triangular matrices regarded as $L$-algebra where $L$ acts trivially on it generates an $L$-variety of almost polynomial growth (see \cite{Kemer1979,GiambrunoRizzo2018}). Clearly another example of algebras generating an $L$-variety of almost polynomial growth is the infinite dimensional Grassmann algebra $G$ where $L$ acts trivially on it (see \cite{Kemer1979,Rizzo2018}). Notice that in the ordinary case Kemer in \cite{Kemer1979} proved that $UT_2$ and $G$ are the only algebras generating varieties of almost polynomial growth.

Recently in \cite{GiambrunoRizzo2018} the authors introduced another algebra with derivations generating a $L$-variety of almost polynomial growth. They considered $UT_2^\varepsilon$ to be the algebra $UT_2$ with the action of the 1-dimensional Lie algebra spanned by the inner derivation $\varepsilon$ induced by $2^{-1}(e_{11}-e_{22})$, where $e_{ij}$'s are the usual matrix units. Also they proved that when the Lie algebra $\Der(UT_2)$ of all derivations acts on $UT_2$, the variety with derivations generated by $UT_2$ has no almost polynomial growth.


Notice that if $\delta$ is the inner derivation of $UT_2$ induced by $2^{-1}e_{12}$, then $\Der(UT_2)$ is a 2-dimensional metabelian Lie algebra with basis $\{\varepsilon,\delta\}$. Thus in order to complete the description of the differential identities of $UT_2$, here we shall study the $T_L$-ideal of the differential identities of $UT_2^\delta$, i.e., the algebra $UT_2$ with the action of the 1-dimensional Lie algebra spanned by $\delta$. In particular we shall prove that $UT_2^\delta$ does not generate an $L$-variety of almost polynomial growth. 

Moreover, we shall study the differential identities of some particular $L$-algebras whose sequence of differential codimension has polynomial growth. In particular we shall exhibit an example of commutative algebra with derivations that generates a $L$-variety of linear growth.

Finally, we shall give an example of infinite dimensional $L$-algebra of exponential growth. We shall present the structure of the differential identities of $\widetilde{G}$, i.e., the infinite dimensional Grassmann algebra with the action of a finite dimensional abelian Lie algebra and we show that, unlike the ordinary case, $\widetilde{G}$ does not generate an $L$-variety of almost polynomial growth.

\section{$L$-algebras and differential identities}
\label{sec:Preliminaries}

Throughout this paper $F$ will denote a field of characteristic zero. Let $A$ be an associative algebra over $F$. Recall that a derivation of $A$ is a linear map $\partial:A\to A$ such that
$$\partial(ab)=\partial(a)b+a\partial(b), \qquad \mbox{ for all } a,b\in A.$$
In particular an inner derivation induced by $ a\in A $ is the derivation $\ad a:A\to A$ of $A$ defined by $(\ad a)(b)=[a,b]=ab-ba$, for all $b\in A$. The set of all derivations of $A$ is a Lie algebra denoted by $\Der(A)$, and the set $\ad (A)$ of all inner derivations of $A$ is a Lie subalgebra of $\Der(A)$.

Let $L$ be a Lie algebra over $F$ acting on $A$ by derivations. If $U(L)$ is its universal enveloping algebra, the $L$-action on $A$ can be naturally extended to an $U(L)$-action. In this case we say that $A$ is an algebra with derivations or an $L$-algebra.


Let $L$ be a Lie algebra. Given a basis $\mathcal{B}=\{h_{i}\ \vert \ i\in I \}$ of the universal enveloping algebra of $L$, $U(L)$, we let $F\langle X|L\rangle$ be the free associative algebra over $F$ with free formal generators $x_{j}^{h_{i}}$, $i\in I$, $j\in\mathbb{N}$. 
We write $x_{i}=x_{i}^{1}$, $1\in U(L)$, and then we set $X=\{x_{1},x_{2},\dots\}$. We let $U(L)$ act on $F\langle X| L\rangle$  by setting
$$\gamma(x_{j_{1}}^{h_{i_{1}}}x_{j_{2}}^{h_{i_{2}}}\dots x_{j_{n}}^{h_{i_{n}}})=x_{j_{1}}^{\gamma h_{i_{1}}}x_{j_{2}}^{h_{i_{2}}}\dots x_{j_{n}}^{h_{i_{n}}}+\dots+x_{j_{1}}^{h_{i_{1}}}x_{j_{2}}^{h_{i_{2}}}\dots x_{j_{n}}^{\gamma h_{i_{n}}},$$
where $\gamma\in L$ and $x_{j_{1}}^{h_{i_{1}}}x_{j_{2}}^{h_{i_{2}}}\dots x_{j_{n}}^{h_{i_{n}}}\in F\langle X|L\rangle$. The algebra $F\langle X|L\rangle$ is called the free associative algebra with derivations on the countable set $X$ and its elements are called differential polynomials (see \cite{GiambrunoRizzo2018,GordienkoKochetov2014,Kharchenko1979}).

Given an $L$-algebra $A$, a polynomial $f(x_{1},\dots,x_{n})\in F\langle X|L\rangle$ is a polynomial identity with derivation of $A$, or a differential identity of $A$, if $f(a_{1},\dots,a_{n})=0$ for all $a_{i}\in A$, and, in this case, we write $f\equiv 0$. 

Let $\Id^{L}(A)=\{f\in F\langle X|L\rangle \ \vert \ f\equiv 0 \mbox{ on } A\}$ be the set of all differential identities of $A$. It is readily seen that $\Id^L (A)$ is a $T_L$-ideal of $F\langle X|L\rangle$, i.e., an ideal invariant under the $U(L)$-action. In characteristic zero every differential identity is equivalent to a system of multilinear differential identities. Hence $\Id^L (A)$ is completely determined by its multilinear polynomial. 

Let
$$P_{n}^{L}=\Span\{x_{\sigma(1)}^{h_{1}}\dots x_{\sigma(n)}^{h_{n}} \ \vert \ \sigma\in S_{n},h_{i}\in \mathcal{B} \}$$
be the space of multilinear differential polynomials in the variables $x_{1},\dots,x_{n}$, $ n\geq 1$. We act on $P_n ^L$ via the symmetric group $S_n$ as follows:  for $\sigma\in S_{n}$, $\sigma(x_{i}^{h})=x_{\sigma(i)}^{h}$. For every $L$-algebra $A$, the vector space $P_{n}^{L}\cap \Id^{L}(A)$ is invariant under this action. Hence the space $P_n ^L (A)= P_n ^L / ( P_{n}^{L}\cap \Id^{L}(A))$ has a structure of left $S_n$-module. The non-negative integer $c_n ^L (A)= \dim P_n ^L (A)$ is called $n$th differential codimension of $A$ and the character $\chi_n ^L (A)$ of $P_n^L (A)$ is called $n$th differential cocharacter of $A$. Since $\Char F=0$, we can write
$$\chi_n^{L}(A)=\sum_{\lambda\vdash n} m^L_{\lambda}\chi_{\lambda},$$
where $\lambda$ is a partition of $n$, $\chi_{\lambda}$ is the irreducible $S_{n}$-character associated to $\lambda$ and $m^L_{\lambda}\geq 0$ is the corresponding multiplicity.



Let $L$ be a Lie algebra and $H$ be a Lie subalgebra of $L$. 
If $A$ is an $L$-algebra, then by restricting the action $A$ can be regarded as a $H$-algebra.
In this case we say that $A$ is an $L$-algebra where $L$ acts on it as the Lie algebra $H$ and we identify the $T_L$-ideal $\Id^L(A)$ and the $T_H$-ideal $\Id^H(A)$, i.e., in $\Id^L(A)$ we omit the differential identities  $x^\gamma \equiv 0$, for all $\gamma \in L\backslash  H$.

Notice that any algebra $A$ can be regarded as $L$-algebra by let $L$ acting on $A$ trivially, i.e., $L$ acts on $A$ as the trivial Lie algebra. Hence the theory of differential identities generalizes the ordinary theory of polynomial identities.

We denote by $P_n$ the space of  multilinear ordinary polynomials in $x_1,\dots,x_n$ and by $\Id(A)$ the $T$-ideal of the free algebra $F\langle X\rangle$ of polynomial identities of $A$. We also write $c_n(A)$ for the $n$th codimension of $A$ and $\chi_n(A)$ for the $n$th cocharacter of $A$. Since the field $F$ is of characteristic zero, we have $\chi_n(A)=\sum_{\lambda\vdash n} m_{\lambda}\chi_{\lambda}$, where $m_\lambda\geq 0$ is the multiplicity of $\chi_\lambda$ in the given decomposition.

Since $U(L)$ is an algebra with unit, we can identify in a natural way $P_n$ with a subspace of $P_n^L$. Hence $P_n\subseteq P^L_n$ and $P_n\cap \Id(A)=P_n\cap \Id^L(A)$.  As a consequence we have the following relations.
\begin{rmk}For all $n\geq 1$,
\label{Rmk Codimensions and multiplicity}
\begin{enumerate}
\item  $c_n(A)\leq c_n^L(A)$;
\item  $m_\lambda\leq m_\lambda^L$, for any $\lambda\vdash n$.
\end{enumerate}
\end{rmk}

Recall that if $A$ is an $L$-algebra then the variety of algebras with derivations generated by $A$ is denoted by $\var ^{L}(A)$ and is called $L$-variety. The growth of $\mathcal{V}= \var ^{L}(A)$ is the growth of the sequence $c_{n}^{L}(\mathcal{V})=c_{n}^{L}(A)$, $n=1,2,\dots$. 

We say that the $L$-variety $\mathcal{V}$ has polynomial growth if $c_{n}^{L}(\mathcal{V})$ is polynomially bounded and $\mathcal{V}$ has almost polynomial growth if  $c_{n}^{L}(\mathcal{V})$ is not polynomially bounded but every proper $L$-subvariety of $\mathcal{V}$ has polynomial growth.

\section{On algebras with derivations of polynomial growth}
\label{sec:Codimension bounded}
In this section we study some algebras with derivations whose sequence of differential codimension has linear growth.

Let first consider the algebra $C= F(e_{11}+e_{22})\oplus Fe_{12} $ where $e_{ij}$'s are the usual matrix units. The Lie algebra $\Der (C)$ of all derivations of $C$ is a 1-dimensional Lie algebra generated by $\varepsilon$ where
$$\varepsilon (\alpha(e_{11}+e_{22})+\beta e_{12})=\beta e_{12},$$
for all $\alpha,\beta \in F$.

Let $C ^\varepsilon$ denote the $L$-algebra $C$ where $L$ acts on it as the Lie algebra $\Der (C)$. Thus for any Lie algebra $L$, we have the following.

\begin{thm}
\label{Thm U^e}
\begin{enumerate}
\item $\Id^L(C^\varepsilon)=\langle [x,y], x^\varepsilon y^\varepsilon , x^{\varepsilon^2}- x ^\varepsilon \rangle_{T_L}$.
\vspace{1mm}

\item $c_n ^L(C^\varepsilon)=n+1$.

\vspace{1mm}

\item $\chi_n^L(C^\varepsilon)=2\chi_{(n)}+\chi_{(n-1,1)}$.
\end{enumerate} 
\end{thm}

\begin{proof}
Let $Q=\langle [x,y], x^\varepsilon y^\varepsilon , x^{\varepsilon^2}- x ^\varepsilon \rangle_{T_L}$. It easily checked that $Q \subseteq \Id^L(C^\varepsilon)$. Since $x^\varepsilon w y^\varepsilon \in Q$, where $w$ is (eventually trivial) monomial of $F\langle X|L\rangle$, we may write $f$, modulo $Q$, as a linear combination of the polynomials
$$x_1 \dots x_n, \; x_k ^\varepsilon x_{i_1}\dots x_{i_{n-1}}, \quad i_1<\dots < i_{n-1}.$$
We next show that these polynomials are linearly independent modulo $\Id^L(C^\varepsilon)$. Suppose that 
$$f=\alpha x_1 \dots x_n+ \sum_{k=1}^{n} \beta_k x_{i_1}\dots x_{i_{n-1}}x_k ^\varepsilon \equiv 0 \; (\bmod P_n ^L \cap \Id^L(C^\varepsilon)).$$
By making the evaluation $x_j=e_{11}+e_{22}$, for all $j=1,\dots,n$, we get $\alpha =0$. Also for fixed $k$, the evaluation $x_k= e_{12}$ and $x_j=e_{11}+e_{22}$ for $j\neq k$ gives $\beta_k=0$. Thus the above polynomials are linearly independent modulo $P_n ^L \cap \Id^L(C^\varepsilon)$. Since $P_n ^L \cap Q \subseteq P_n ^L \cap \Id^L(C^\varepsilon)$, this proves that $\Id^L(C^\varepsilon)=Q$ and the above polynomials are a basis of $P_n ^L$ modulo $P_n ^L \cap \Id^L(C^\varepsilon)$. Hence $c_n ^L(C^\varepsilon)=n+1$.

We now determine the decomposition of the $n$th differential cocharacter of this algebra. Suppose that $\chi_{n}^{L}(C^\varepsilon)=\sum_{\lambda\vdash n} m_{\lambda}\chi_{\lambda}$. Let consider the standard tableau
$$T_{(n)}=\begin{array}{|c|c|c|c|}\hline
1 & 2 & \dots &n \\ \hline
\end{array}\;$$
and the monomials
\begin{equation}
\label{hwv of T_(n)}
f_{(n)}=x^n, \qquad f_{(n)}^\varepsilon =x^\varepsilon x^{n-1}
\end{equation}
obtained from the essential idempotents corresponding to the tableau $T_{(n)}$ by identifying all the elements in the row. Clearly $f_{(n)}$ and $f_{(n)}^\varepsilon$ are not identities of $C^\varepsilon$. Moreover, they are linear independent modulo $\Id^L(C^\varepsilon)$. In fact, suppose that $\alpha f_{(n)} + \beta f_{(n)}^\varepsilon \equiv 0 (\bmod \; \Id^L(C^\varepsilon))$. By making the evaluation $x=e_{11}+e_{22}$ we get $\alpha=0$. Moreover, if we evaluate $x=e_{11}+e_{22}+e_{12}$, we obtain $\beta=0$. Thus it follows that $m_{(n)}\geq 2$.

Since $\deg \chi_{(n)}=1$ and $\deg \chi_{(n-1,1)}=n-1$, if we find a differential polynomial corresponding to the partition $(n-1,1)$ which is not a differential identity of $C^\varepsilon$, we may conclude that $\chi_n^L(C^\varepsilon)=2\chi_{(n)}+\chi_{(n-1,1)}$. 

Let consider the polynomial
$$f_{(n-1,1)}=(x^\varepsilon y - y^\varepsilon x)x^{n-2}$$
obtained from the essential idempotent corresponding to the standard tableau
\begin{equation*}
T_{(n-1,1)}= \begin{array}{|c|c|c|c|}\hline
1 & 3 & \dots & n \\ \hline
2 \\ \cline{1-1}
\end{array}\:
\end{equation*}
by identifying all the elements in each row of the tableau. Evaluating $x=e_{11}+e_{22}$ and $y=e_{12}$ we get $f_{(n-1,1)}=-e_{12}\neq 0$ and $f_{(n-1,1)}$ is not a differential identity of $C^\varepsilon$. Thus the claim is proved.
\end{proof}

\

Let now consider the algebra $M_1=F e_{22} \oplus Fe_{12}$ and let $\varepsilon$ and $\delta$ be  derivations of $M_1$ such that
\begin{equation}
\label{Derivation of M_1}
\varepsilon (\alpha e_{22}+ \beta e_{12})= \beta e_{12}, \qquad \delta (\alpha e_{22}+ \beta e_{12})= \alpha e_{12},
\end{equation}
for all $\alpha, \beta \in F$.
\begin{Lem}
$\Der(M_1)$ is a $2$-dimensional metabellian Lie algebra spanned by $\varepsilon$ and $ \delta$ defined in \eqref{Derivation of M_1}.
\end{Lem}

\begin{proof}
Let consider the Lie algebra $D$ spanned by $\varepsilon$ and $\delta$. Since $[\varepsilon,\delta]=\delta$, $D$ is a $2$-dimensional metabelian Lie algebra and $D\subseteq \Der(M_1)$.

Let now consider $\gamma\in \Der(M_1)$. Notice that
$\gamma(e_{22}e_{12})=\gamma(e_{22})e_{12}+e_{22}\gamma(e_{12})=e_{22}\gamma(e_{12}).$
Since $\gamma(e_{22}e_{12})=0$, it follows that 
$$\gamma(e_{12})=\alpha e_{12},$$
for some $\alpha\in F$. On the other hand, $\gamma(e_{12})=\gamma(e_{12}e_{22})=\alpha e_{12}+e_{12} \gamma(e_{22})$. Thus it follows that $e_{12}\gamma(e_{22})=0$. Hence 
$$\gamma(e_{22})=\beta e_{12},$$
for some $\beta \in F$. Thus we have that $\gamma =\alpha \varepsilon + \beta \delta\in D$ and the claim is proved.
\end{proof}

Similarly, if we consider the algebra $M_2=F e_{11} \oplus Fe_{12}$ and we assume that $\varepsilon$ and $\delta$ are derivation of $M_2$ such that
\begin{equation}
\label{Derivation of M_2}
\varepsilon (\alpha e_{11}+ \beta e_{12})= \beta e_{12}, \qquad \delta (\alpha e_{11}+ \beta e_{12})= \alpha e_{12},
\end{equation}
for all $\alpha,\beta \in F$, then we have the following.

\begin{Lem}
$\Der(M_2)$ is a $2$-dimensional metabellian Lie algebra spanned by $\varepsilon$ and $ \delta$ defined in \eqref{Derivation of M_2}.
\end{Lem}

Let $L$ be any Lie algebra. We shall denote by $M_1$ and $M_2$ the $L$-algebras $M_1$ and $M_2$ where $L$ acts trivially on them. Since $x^\gamma\equiv 0$ for all $\gamma \in L$, in this case we are dealing with ordinary identities. Thus we have the following result.

\begin{thm}{\cite[Lemma 3]{GiambrunoLaMattina2005}}
\label{thmM}
\begin{enumerate}
\item $Id^L(M_1)=\langle x[y,z]\rangle_{T_L}$ and $Id^L(M_2)=\langle[x,y] z \rangle_{T_L}$.
\vspace{1mm}

\item $c_n ^L(M_1)=c_n ^L(M_2)=n$.
\vspace{1mm}

\item $\chi_n^L(M_1)=\chi_n^L(M_2)=\chi_{(n)}+\chi_{(n-1,1)}$.
\end{enumerate}
\end{thm}

Let now denote by $M_1 ^\varepsilon$ and $M_2^\varepsilon$ the $L$-algebras $M_1$ and $M_2$ where $L$ acts on them as the $1$-dimensional Lie algebra spanned by the derivation $\varepsilon$ defined in \eqref{Derivation of M_1} and \eqref{Derivation of M_2}, respectively.

\begin{thm} \label{thmM^e}
\begin{enumerate}
\item $Id^L(M_1^\varepsilon)=\langle  xy^{\varepsilon}, x^\varepsilon y- y^\varepsilon x -[x,y], x^{\varepsilon^{2}}-x^{\varepsilon}\rangle_{T_{L}}$ and $Id^L(M_2^\varepsilon)=\langle  x^{\varepsilon}y, x y^\varepsilon - y x^\varepsilon - [x,y], x^{\varepsilon^{2}}-x^{\varepsilon}\rangle_{T_{L}}$.

\vspace{1mm}
\item  $c_n ^L(M_1 ^\varepsilon)=c_n ^L(M_2^\varepsilon)=n+1$.

\vspace{1mm}
\item $\chi_n^L(M_1^\varepsilon)=\chi_n^L(M_2^\varepsilon)=2\chi_{(n)}+\chi_{(n-1,1)}$.
\end{enumerate}
\end{thm}
\begin{proof}
If $Q$ is the $T_L$-ideal generated by the polynomials $xy^{\varepsilon}, x^\varepsilon y- y^\varepsilon x -[x,y], x^{\varepsilon^{2}}-x^{\varepsilon}$, then it easy to check that $Q\subseteq \Id^L (M_1 ^\varepsilon)$.

Since $x^\varepsilon y^\varepsilon, x[y,z]\in Q$, the polynomials
$$x_j x_{i_1}\dots x_{i_n-1}, \; x_1 ^\varepsilon x_2 \dots x_n,  \quad i_1<\dots < i_{n-1},$$
span $P_n ^L$ modulo $P_n ^L \cap Q$ and we claim that they are linearly independent modulo $\Id^L (M_1 ^\varepsilon)$. In fact, let $f\in P_n ^L \cap \Id^L (M_1 ^\varepsilon)$ be a linear combination of these polynomials, i.e.,
$$f=\sum_{j=1}^{n} \alpha_j x_j x_{i_1}\dots x_{i_n-1} + \beta x_1 ^\varepsilon x_2 \dots x_n \equiv 0 \; (\bmod\; P_n ^L \cap \Id^l (M_1 ^\varepsilon)).$$
For fixed $j \neq 1$, from the substitutions $x_j=e_{12}$ and $x_k=e_{22}$ for $k\neq j$ we get $\alpha_j=0$, $j\neq 1$. By making the evaluation $x_k=e_{22}$ for all $k=1,\dots, n$, we obtain $\alpha_1=0$. Finally by evaluating $x_1= e_{12}$ and $x_k = e _{22}$ for $k\neq 1$, we get $\beta =0$. Thus the above polynomials are linearly independent modulo $P_n ^L \cap \Id^l (M_1 ^\varepsilon)$. Since $P_n ^L \cap Q \subseteq P_n ^L \cap \Id^L(M_1^\varepsilon)$, this proves that $\Id^L(M_1^\varepsilon)=Q$ and the above polynomials are a basis of $P_n ^L$ modulo $P_n ^L \cap \Id^L(M_1^\varepsilon)$. Clearly $c_n ^L(M_1 ^\varepsilon)=n+1$.

We now determine the decomposition of the $n$th differential cocharacter $\chi_n^L(M_1^\varepsilon)$ of this algebra. Suppose that $\chi_{n}^{L}(M_1^\varepsilon)=\sum_{\lambda\vdash n} m_{\lambda}\chi_{\lambda}$. We consider the tableau $T_{(n)}$ defined in Theorem \ref{Thm U^e} and let $f_{(n)}$ and $f_{(n)}^\varepsilon$ be the corresponding polynomials defined in \eqref{hwv of T_(n)}. It is clear that $f_{(n)}$ and $f_{(n)}^\varepsilon$ are not identities of $M_1^\varepsilon$. Moreover, they are linear independent modulo $\Id^L(M_1^\varepsilon)$. In fact, suppose that $\alpha f_{(n)} + \beta f_{(n)}^\varepsilon \equiv 0 (\bmod \; \Id^L(M_1^\varepsilon))$. By making the evaluation $x=e_{22}$ we get $\alpha=0$. Moreover, if we evaluate $x=e_{22}+e_{12}$, we obtain $\beta=0$. Thus it follows that $m_{(n)}\geq 2$. By Remark \ref{Rmk Codimensions and multiplicity} and Theorem \ref{thmM} we have $m_{(n-1,1)}\geq 1$. Thus, since $\deg \chi_{(n)}=1$ and $\deg \chi_{(n-1,1)}=n-1$, it follows that $\chi_n^L(M_1^\varepsilon)=2\chi_{(n)}+\chi_{(n-1,1)}$.

A similar proof holds for the algebra $M_2 ^\varepsilon$.
\end{proof}

Let $M_1 ^\delta$ and $M_2^\delta$ be the $L$-algebras $M_1$ and $M_2$ where $L$ acts on them as the $1$-dimensional Lie algebra spanned by the derivation $\delta$ defined in \eqref{Derivation of M_1} and \eqref{Derivation of M_2}, respectively.
We do not present the proof of the next theorem since it is very similar to the proof of the
previous theorem.

\begin{thm}
\begin{enumerate}
\item $Id^L(M_1^\delta)=\langle x[y,z],x^\delta y - y^\delta x, xy^{\delta}, x^{\delta^{2}}\rangle_{T_{L}}$ and $Id^L(M_2^\delta)=\langle [x,y]z, x y^\delta - y x^\delta, x^{\delta}y, x^{\delta^{2}}\rangle_{T_{L}}$.
\vspace{1mm}

\item  $c_n ^L(M_1 ^\delta)=c_n ^L(M_2^\delta)=n+1$.
\vspace{1mm}

\item $\chi_n^L(M_1^\delta)=\chi_n^L(M_2^\delta)=2\chi_{(n)}+\chi_{(n-1,1)}$.
\end{enumerate}
\end{thm}

Let now $L$ be a $2$-dimensional metabelian Lie algebra. Let denote by $M_1^D$ the $L$-algebra $M_1$ where $L$ acts on it as the Lie algebra $\Der(M_1)$ and $M_2^D$ the $L$-algebra $M_2$ where $L$ acts on it as the Lie algebra $\Der(M_2)$.  

\begin{rmk}
\begin{enumerate}
\item $x^\delta y - y^\delta x \in \langle  xy^{\varepsilon}, x^\varepsilon y- y^\varepsilon x -[x,y],  x^{ \varepsilon\delta}- x^{\delta}\rangle_{T_{L}}$.

\vspace{1mm}
\item $x y^\delta - y x^\delta  \in \langle  x^{\varepsilon}y, x y^\varepsilon - y x^\varepsilon - [x,y], x^{ \varepsilon\delta}- x^{\delta} \rangle_{T_{L}}$.
\end{enumerate}

\end{rmk}

\begin{proof}
First notice that $[x,y]^\delta \in \langle  xy^{\varepsilon}, [x,y]^\varepsilon -[x,y]\rangle_{T_{L}}$. Thus, since $[x,y]^\varepsilon \equiv [x,y] (\bmod \; \langle  xy^{\varepsilon}, x^\varepsilon y- y^\varepsilon x -[x,y]\rangle_{T_{L}})$, it follows that
$$[x,y]^\delta \in \langle  xy^{\varepsilon}, x^\varepsilon y- y^\varepsilon x -[x,y],  x^{ \varepsilon\delta}- x^{\delta}\rangle_{T_{L}}.$$
Moreover, since $xy^\delta \in \langle  xy^{\varepsilon},  x^{ \varepsilon\delta}- x^{\delta}\rangle_{T_{L}}$, we get
$$x^\delta y - y^\delta x \in \langle  xy^{\varepsilon}, x^\varepsilon y- y^\varepsilon x -[x,y],  x^{ \varepsilon\delta}- x^{\delta}\rangle_{T_{L}}.$$
A similar proof holds for the other statement.
\end{proof}

By following closely the proof of the Theorem \ref{thmM^e}, taking into account the due changes, we get the following.

\begin{thm}
\begin{enumerate}
\item $Id^L(M_1^D)=\langle  xy^{\varepsilon}, x^\varepsilon y- y^\varepsilon x -[x,y], x^{\varepsilon^{2}}-x^{\varepsilon}, x^{\delta\varepsilon}, x^{ \varepsilon\delta}- x^{\delta}\rangle_{T_{L}}$ and $Id^L(M_2^D)=\langle  x^{\varepsilon}y, x y^\varepsilon - y x^\varepsilon - [x,y], x^{\varepsilon^{2}}-x^{\varepsilon}, x^{\delta\varepsilon}, x^{ \varepsilon\delta}- x^{\delta} \rangle_{T_{L}}$.

\vspace{1mm}
\item   $c_n ^L(M_1 ^D)=c_n ^L(M_2^D)=n+2$.
\vspace{1mm}

\item $\chi_n^L(M_1^D)=\chi_n^L(M_2^D)=3\chi_{(n)}+\chi_{(n-1,1)}$.
\end{enumerate}
\end{thm}

\section{The algebra of $2\times 2$ upper triangular matrices and its differential identities}

In this section we study the growth of differential identities of the algebra $UT_2$ of $2 \times 2$ upper triangular matrices over $F$.

%

Let $L$ be any Lie algebra over $F$ and let denote by $UT_2$ the $L$-algebra $UT_2$ where $L$ acts trivially on it. Since $x^\gamma \equiv 0$, for all $\gamma \in L$, is a differential identity of $UT_2$, we are dealing with ordinary identities. Thus by \cite{Malcev1971}, \cite{Kemer1979} and by the proof of Lemma $3.5$ in \cite{BenantiGiambrunoSviridova2004}, we have the following results.

\begin{thm}
\label{ThmIdCnOrdinaryUT2}
\begin{enumerate}
\item $\Id^L(UT_{2})=\langle [x_{1},x_{2}][x_{3},x_{4}]\rangle_{T_L}$.

\vspace{1mm}
\item $c^L _{n}(UT_{2})=2^{n-1}(n-2)+2.$

\vspace{1mm}
		
\item If $\chi_{n}^{L}(UT_{2})=\sum_{\lambda\vdash n} m_{\lambda} \chi_{\lambda}$ is the $n$th differential cocharacter of $UT_{2}$, then
$$m_{\lambda}=\begin{cases}
		1, & \mbox{ if } \lambda=(n) \\
		q+1, & \mbox{ if } \lambda=(p+q,p) \mbox{ or } \lambda=(p+q,p,1) \\
		0 & \mbox{ in all other cases}
		\end{cases}.$$
\end{enumerate}
\end{thm}

\begin{thm}
\label{Thm:almost polynomial growth $UT_2$}
$\var^L(UT_2)$ has almost polynomial growth.
\end{thm}


\

Let now $\varepsilon$ be the inner derivation of $UT_2$ induced by $2^{-1}(e_{11}-e_{22})$, i.e.,
 $$\varepsilon(a)=2^{-1}[e_{11}-e_{22},a],~ \mbox{for~ all~} a\in UT_2,
$$
where $e_{ij}$'s are the usual matrix units. We shall denote by $UT_2^\varepsilon$ the $L$-algebra $UT_2$ where $L$ acts on it as the $1$-dimensional Lie algebra spanned by $\varepsilon$.
In \cite{GiambrunoRizzo2018} the authors proved the following.

\begin{thm}\label{UT2e}{\cite[Theorems 5 and 12]{GiambrunoRizzo2018}}
	\begin{enumerate}
		\item $\Id^{L}(UT_2^{\varepsilon})=\langle [x,y]^{\varepsilon}-[x,y], x^{\varepsilon}y^{\varepsilon}, x^{\varepsilon^{2}}-x^{\varepsilon}\rangle_{T_{L}}$.
		
		\vspace{1mm}
		\item $c_n^{L}(UT_2^{\varepsilon})=2^{n-1}n+1$.
		
		\vspace{1mm}
		\item If $\chi_{n}^{L}(UT_{2}^\varepsilon)=\sum_{\lambda\vdash n} m_{\lambda}^\varepsilon \chi_{\lambda}$ is the $n$th differential cocharacter of $UT_{2}^{\varepsilon}$, then:
		$$m_{\lambda}^{\varepsilon}=\begin{cases}
		n+1, & \mbox{ if } \lambda=(n) \\
		2(q+1), & \mbox{ if } \lambda=(p+q,p) \\
		q+1, & \mbox{ if } \lambda=(p+q,p,1) \\
		0 & \mbox{ in all other cases}
		\end{cases}.$$
	\end{enumerate}
	
\end{thm}

\begin{thm}{\cite[Theorem 15]{GiambrunoRizzo2018}}
\label{Thm:almost polynomial growth $UT_2^epsilon$}
$\var^L(UT_2^\varepsilon)$ has almost polynomial growth.
\end{thm}


\

Let now $\delta$ be the inner derivation of $UT_2$ induced by $2^{-1}e_{12}$, i.e., $$\delta(a)=2^{-1}[e_{12},a],~ \mbox{for~ all~} a\in UT_2.
$$
Let denote by $UT_2^\delta$ the $L$-algebra $UT_2$ where $L$ acts as the 1-dimensional Lie algebra spanned by $\delta$. 
The following remarks are easily verified.
\begin{rmk}
\label{rmkIdAllDerivations}
$[x,y][z,w]\equiv 0$, $[x,y]^{\delta}\equiv 0 $, $x^\delta y^\delta \equiv 0$, $x^\delta [y,z] \equiv 0 $ and $x^{\delta^2} \equiv 0$ are differential identities of $UT_{2}^{\delta}$.
\end{rmk}

\begin{rmk}
\label{rmkConseguenceAllDerivations}
$ x^{\delta}y[z,w],[x,y]zw^{\delta}, x^{\delta}yz^{\delta}\in\langle x^\delta y^\delta , x^\delta [y,z], [x,y]^{\delta} \rangle_{T_{L}} $.
\end{rmk}

\begin{rmk}
\label{rmkConseguenceAllDerivations2}
For any permutations $\sigma\in S_{t}$, we have
$$[x^{\delta}_{\sigma(1)},x_{\sigma(2)},\dots,x_{\sigma(t)}]\equiv [x^{\delta}_{1},x_{2},\dots,x_{t}]\hspace{0.20cm} (\bmod\hspace{0.07cm} \langle x^{\delta}[y,z],[x,y]^{\delta}\rangle_{T_{L}}).$$
\end{rmk}
\begin{proof}
Let $u_{1}, u_{2},u_{3}$ be monomials. We consider $w=u_{1}x_{i}x_{j}u_{2}y^{\delta}u_{3}$. Since $x_{i}x_{j}=x_{j}x_{i}+[x_{i},x_{j}]$, it follows that $w\equiv u_{1}x_{j}x_{i}u_{2}y^{\delta}u_{3}\hspace{0.20cm} (\bmod\hspace{0.07cm}  \langle x^{\delta}[y,z],[x,y]^{\delta} \rangle_{T_{L}} )$. In the same way we can show that $u_{1}y^{\delta}u_{2}z_{i}z_{j}u_{3}\equiv u_{1}y^{\delta}u_{2}z_{j}z_{i}u_{3}\hspace{0.20cm} (\bmod\hspace{0.07cm}  \langle x^{\delta}[w,z]\rangle_{T_{L}} )$. Hence in every monomial
$$x_{i_{1}}\dots x_{i_{t}}y^{\delta}z_{j_{1}}\dots z_{j_{p}}$$
we can reorder the variables to the left and to the right of $y^{\delta}$. Since $[x,y]^{\delta}=[x^{\delta},y]-[y^{\delta},x]$, we can reorder all the variables in any commutator $[x^{\delta}_{i_{1}},x_{i_{2}},\dots,x_{i_{t}}]$ as claimed.
\end{proof}

\begin{Lem}
\label{Lemmma generatori UT^d}
The $T_L$-ideal of identities of $UT_2^\delta$ is generated by the following polynomials
$$[x,y][z,w],\; [x,y]^{\delta},\; x^\delta [y,z],\; x^{\delta}y^{\delta},\; x^{\delta^{2}}.$$
%
\end{Lem}

\begin{proof}
Let $ Q=\langle [x,y][z,w],[x,y]^{\delta},x^\delta [y,z], x^{\delta}y^{\delta}, x^{\delta^{2}}\rangle_{T_{L}}$. By Remark \ref{rmkIdAllDerivations}, $ Q\subseteq \Id^{L}(UT_{2}^\delta) $. 

By the Poincar\'{e}-Birkhoff-Witt Theorem (see \cite{Procesi2006}) every differential multilinear polynomial in $x_{1}, \dots, x_{n}$ can be written as a linear combination of products of the type
\begin{equation}
\label{eqP-B-W}
x_{i_{1}}^{\alpha_{1}}\dots x_{i_{k}}^{\alpha_{k}} w_{1}\dots w_{m}
\end{equation}
where $\alpha_{1},\dots,\alpha_{k}\in U(L)$, $w_{1}\dots,w_{m}$ are left normed commutators in the $x_{i}^{\alpha_{j}}$s, $\alpha_{j}\in U(L)$, and $i_{1}<\dots<i_{k}$. Since $[x_{1}^{\alpha_{1}},x_{2}^{\alpha_{2}}][x_{3}^{\alpha_{3}},x_{4}^{\alpha_{4}}]\in Q$, with $\alpha_{1},\alpha_{2},\alpha_{3},\alpha_{4}\in\{1,\delta\}$, then, modulo $\langle [x_{1}^{\alpha_{1}},x_{2}^{\alpha_{2}}][x_{3}^{\alpha_{3}},x_{4}^{\alpha_{4}}], x^{\delta^{2}}\rangle_{T_{L}}$, in \eqref{eqP-B-W} we have $\alpha_{j}\in\{1,\delta\}$ and $m\leq 1$, so, only at most one commutator can appear in \eqref{eqP-B-W}. Thus by Remark \ref{rmkConseguenceAllDerivations} every multilinear monomial in $P_{n}^{L}$ can be written, modulo $Q$, as linear combination of the elements of the type
$$x_{1}\dots x _{n},\quad x_{h_{1}}\dots x_{h_{n-1}}x_{j}^{\delta}, \quad x_{i_{1}}\dots x_{i_{k}}[x_{j_{1}}^{\gamma},x_{j_{2}},\dots,x_{j_{m}}],$$
where $ h_{1}<\dots<h_{n-1} $, $ i_{1}<\dots<i_{k} $, $m+k=n$, $m\geq 2$, $\gamma\in\{1,\delta\}$. 

Let us now consider the left normed commutators $[x_{j_{1}}^{\gamma},x_{j_{2}},\dots,x_{j_{m}}]$ and suppose first that $\gamma=1$. Since $[x_{1},x_{2}][x_{3},x_{4}]\in Q$, then by Theorem \ref{ThmIdCnOrdinaryUT2}
$$[x_{j_{1}},x_{j_{2}},\dots,x_{j_{m}}]\equiv [x_{k},x_{h_{1}},\dots,x_{h_{m-1}}]\hspace{0.20cm} (\bmod\hspace{0.07cm}  Q),$$
where $k>h_{1}<\dots<h_{m-1}$.

Suppose now $\gamma=\delta$, then by Remark \ref{rmkConseguenceAllDerivations2} we get
$$[x_{j_{1}}^{\delta},x_{j_{2}},\dots,x_{j_{m}}] \equiv [x^{\delta}_{1},x_{2},\dots,x_{t}]\hspace{0.20cm} (\bmod\hspace{0.07cm} \langle x^{\delta}[y,z],[x,y]^{\delta}\rangle_{T_{L}}).$$
It follows that $P_{n}^{L}$ is spanned, modulo $P_{n}^{L}\cap Q$, by the polynomials
\begin{align}
&x_{1}\dots x_{n}, \quad x_{i_{1}}\dots x_{i_{m}}[x_{k},x_{j_{1}},\dots,x_{j_{n-m-1}}],\nonumber \\
&x_{h_{1}}\dots x_{h_{n-1}}x_{r}^{\delta},\quad x_{i_{1}}\dots x_{i_{m}}[x_{l_{1}}^{\delta},x_{l_{2}},\dots,x_{l_{n-m}}],\label{EqBaseNoIdentityAllDerivations}
\end{align}
where $i_{1}<\dots<i_{m}$, $k>j_{1}<\dots<j_{n-m-1}$, $h_{1}<\dots<h_{n-1}$, $l_{1}<\dots<l_{n-m}$, $m\neq n-1,n $.

Next we show that these polynomials are linearly independent modulo $\Id^{L}(UT_{2} ^\delta)$. Let $I=\{i_{1},\dots,i_{m}\}$ be a subset of $\{1,\dots,n\}$ and $k\in  \{1,\dots,n\}\setminus I$ such that $k>\min ( \{1,\dots,n\}\setminus I)$,then set $X_{I,k}=x_{i_{1}}\dots x_{i_{m}}[x_{k},x_{j_{1}},\dots,x_{j_{n-m-1}}]$. Also for $I^{'}=\{i_{1},\dots,i_{m}\}\subseteq\{1,\dots,n\}$, $0\leq \vert I^{'}\vert< n-1$, set $X_{I^{'}}^{\delta}= x_{i_{1}}\dots x_{i_{m}}[x_{l_{1}}^{\delta},x_{l_{2}},\dots,x_{l_{n-m}}]$ and suppose that
\begin{align*}
f=\sum_{I,J} \alpha_{I,k} X_{I,k}+\sum_{I^{'}} \alpha_{I^{'}}^{\delta} X_{I^{'}}^{\delta}+&\sum_{k=1}^{n}\alpha_{r}^{\delta}x_{h_{1}}\dots x_{h_{n-1}}x_{r}^{\delta}\\
&+\beta x_{1}\dots x_{n}\equiv 0\hspace{0.20cm} (\bmod\hspace{0.07cm} P_{n}^{L}\cap \Id^{L}(UT_{2}^\delta)).
\end{align*}
In order to show that all coefficients $\alpha_{I,k}$, $\alpha_{I^{'}}^{\delta}$, $\alpha_{r}^{\delta}$, $\beta$ are zero we will make some evaluations. If we evaluate $x_{1}=\dots=x_{n}=e_{11}+e_{22}$ we get $\beta=0$. For a fixed $r$, by setting $x_{h_{1}}=\dots=x_{h_{n-1}}=e_{11}+e_{22}$ and $x_{r}=e_{22}$ we get $\alpha_{r}^{\delta}=0$. Also, for a fixed $I^{'}=\{i_{1},\dots,i_{m}\}$, by making the evaluations $x_{i_{1}}=\dots=x_{i_{m}}=e_{11}+e_{22}$, $x_{l_{1}}=\dots=x_{l_{n-m}}=e_{22}$ we obtain $\alpha_{I^{'}}^{\delta}=0$. Finally, for fixed $I=\{i_{1},\dots,i_{m}\}$ and $ J=\{j_{1},\dots,j_{n-m-1}\} $, from the substitutions $x_{i_{1}}=\dots=x_{i_{m}}=e_{11}+e_{22}$, $x_{k}=e_{12}$, $x_{j_{1}}=\dots=x_{j_{n-m-1}}=e_{22}$, it follows that $\alpha_{I,k}=0$.

We have proved that $\Id^{L}(UT_{2}^\delta)=Q$ and the elements in \eqref{EqBaseNoIdentityAllDerivations} are a basis of $P_{n}^{L}$ modulo $P_{n}^{L}\cap \Id^{L}(UT_{2}^\delta)$.
\end{proof}

We now compute the $n$th differential cocharacter of $UT_2^\delta$. Write
\begin{equation}
\label{CocharacterDer}
\chi_{n}^{L}(UT_{2}^\delta)=\sum_{\lambda\vdash n} m_{\lambda}^\delta \chi_{\lambda}.
\end{equation}
In the following lemmas we compute the non-zero multiplicities of such cocharacter.

%

\begin{Lem}
\label{LemMoltRiga delta}
In \eqref{CocharacterDer}  $m_{(n)}^\delta \geq n+1$.
\end{Lem}

\begin{proof}
We consider the following tableau:
$$T_{(n)}=\begin{array}{|c|c|c|c|}\hline
1 & 2 & \dots &n \\ \hline
\end{array}\;.$$
We associate to $T_{(n)}$ the monomials
\begin{equation}
\label{a}
a(x)=x^{n},
\end{equation}
\begin{equation}
\label{a^epsilon}
a_{k}^{(\delta)}(x)=x^{k-1}x^{\delta}x^{n-k},
\end{equation}
for all $k=1,\dots,n$. These monomials are obtained from the essential idempotents corresponding to the tableau $T_{(n)}$ by identifying all the elements in the row. It is easily checked that $a(x)$, $a_{k}^{(\delta)}(x)$, $k=1,\dots,n$, do not vanish in $UT_{2}^{\delta}$.

Next we shall prove that the $n+1$ monomials $a(x)$, $a_{k}^{(\delta)}(x)$, $k=1,\dots,n$, are linearly independent modulo $\Id^{L}(UT_{2}^\delta)$. In fact, suppose that
$$\alpha a(x)+\sum_{k=1}^{n}\alpha^{\delta}_{k} a_{k}^{(\delta)}(x) \equiv 0\hspace{0.20cm} (\bmod\hspace{0.07cm}  \Id^{L}(UT_{2}^\delta)).$$
By setting $x=e_{11}+e_{22}$ it follows that $\alpha=0$. Moreover, if we substitute $x=\beta e_{11}+e_{22}$ where $\beta\in F$, $\beta\neq 0$, we get $ \sum_{k=1}^{n} (1-\beta)\beta^{k-1}\alpha^{\delta}_{k}=0$. Since $|F|=\infty$, we can choose $\beta_{1},\dots,\beta_{n}\in F$, where $\beta_{i}\neq 0$ and $\beta_{i}\neq\beta_{j}$, for all $1\leq i\neq j\leq n$. Then we get the following homogeneous linear system of $n$ equations in the $n$ variables $\alpha^{\delta}_{k}$, $k=1,\dots,n$,
\begin{equation}
\label{SistemaTabella1rigaAllDerivatios}
\sum_{k=1}^{n}\beta_{i}^{k-1}\alpha^{\delta}_{k}=0,\quad i=1,\dots,n.
\end{equation}
Since the matrix associated to the system \eqref{SistemaTabella1rigaAllDerivatios} is a Vandermonde matrix, it follows that $\alpha^{\delta}_{k}=0$, for all $k=1,\dots,n$. Thus the monomials $a(x)$,  $a_{k}^{(\delta)}(x)$, $k=1,\dots,n$, are linearly independent modulo $\Id^{L}(UT_{2}^\delta)$. This says that $m_{(n)}^\delta \geq n+1$.
\end{proof}

\begin{Lem}
\label{LemMolt2Righe delta}
Let $p\geq 1$ and $q\geq 0$. If $\lambda=(p+q,p)$ then in \eqref{CocharacterDer} we have $m_{\lambda}^\delta \geq 2(q+1)$.
\end{Lem}

\begin{proof}
For every $i=0,\dots,q$ we define $T_{\lambda}^{(i)}$ to be the tableau
\begin{small}
\begin{equation*}
\begin{array}{|c|c|c|c|c|c|c|c|c|c|c|}\hline
i+1 & i+2 & \dots & i+p-1 & i+p & 1 & \dots & i & i+2p+1 & \dots & n \\ \hline
i+p+2 & i+p+3 & \dots & i+2p & i+p+1\\ \cline{1-5}
\end{array}\:.
\end{equation*}
\end{small}
We associate to $T_{\lambda}^{(i)}$ the polynomials
\begin{equation}
\label{b^p,q}
b^{(p,q)}_{i}(x,y)=x^{i}\underbrace{\overline{x}\dots\widetilde{x}}_{p-1}[x,y]\underbrace{\overline{y}\dots\widetilde{y}}_{p-1}x^{q-i},
\end{equation}
\begin{equation}
\label{b^p,q,epsilon}
b^{(p,q,\delta)}_{i}(x,y)=x^{i}\underbrace{\overline{x}\dots\widetilde{x}}_{p-1}(x^{\delta}y-y^{\delta}x)\underbrace{\overline{y}\dots\widetilde{y}}_{p-1}x^{q-i},
\end{equation}
where the symbols $-$ or $\thicksim$ means alternation on the corresponding variables. The polynomials $b^{(p,q)}_{i}$, $b^{(p,q,\delta)}_{i}$ are obtained from the essential idempotents corresponding to the tableau $T_{\lambda}^{(i)}$ by identifying all the elements in each row of the tableau. It is clear that $b^{(p,q)}_{i}$, $b^{(p,q,\delta)}_{i}$, $i=0,\dots,q$, are not differential identities of $UT_{2}^{\delta}$. We shall prove that the above $2(q+1)$ polynomials are linearly independent modulo $\Id^{L}(UT_{2}^\delta)$. Suppose that
$$\sum_{i=0}^{q}\alpha_{i}b^{(p,q)}_{i}+\sum_{i=0}^{q}\alpha^{\delta}_{i}b^{(p,q,\delta)}_{i}\equiv 0 \hspace{0.20cm} (\bmod\hspace{0.07cm}  \Id^{L}(UT_{2}^\delta)).$$
If we set $x=\beta e_{11}+e_{22}$, with $\beta\in F$, $\beta\neq 0$, and $y=e_{11}$, we obtain
$$\sum_{i=0}^{q}(-1)^{p-1}\beta^{i} \alpha^{\delta}_{i}=0.$$
Since $|F|=\infty$, we can take $\beta_{1},\dots,\beta_{q+1}\in F$, where $\beta_{j}\neq 0$, $\beta_{j}\neq\beta_{k}$, for all $1\leq j\neq k\leq q+1$. Then we obtain the following homogeneous linear system of $q+1$ equations in the $q+1$ variables $\alpha^{\delta}_{i}$, $i=0,\dots,q$,
\begin{equation}
\label{SistemaTabella2riga}
\sum_{i=0}^{q}\beta_{j}^{i}\alpha^{\delta}_{i}=0,\quad j=1,\dots,q+1.
\end{equation}

Since the matrix of this system is a Vandermonde matrix, it follows that $\alpha^{\delta}_{i}=0$, for all $i=0,\dots,q$. Hence we may assume that the following identity holds
$$\sum_{i=0}^{q}\alpha_{i}b^{(p,q)}_{i}\equiv 0 \hspace{0.20cm} (\bmod\hspace{0.07cm}  \Id^{L}(UT_{2}^\delta)).$$

If we evaluate $x=\beta e_{11}+e_{12}+e_{22}$, where $\beta\in F$,  $\beta\neq 0$, and $y=e_{11}$, then we get
\begin{equation}
\label{EquTabella2righe}
\sum_{i=0}^{q}(-1)^{p-1}\beta^{i}\alpha_{i}=0.
\end{equation}
Since $|F|=\infty$, we choose $\beta_{1},\dots,\beta_{q+1}\in F$, where $\beta_{j}\neq 0$, $\beta_{j}\neq\beta_{k}$, for all $1\leq j\neq k\leq q+1$. Then from \eqref{EquTabella2righe} we obtain a homogeneous linear system of $q+1$ equations in the $q+1$ variables $\alpha_{i}$, $i=0,\dots,q$, equivalent to the linear system \eqref{SistemaTabella2riga}. Therefore $\alpha_{i}=0$, for all $i=0,\dots,q$. Hence the polynomials $b^{(p,q)}_{i}$, $b^{(p,q,\delta)}_{i}$, $i=0,\dots,q$, are linearly independent modulo $\Id^{L}(UT_{2}^\delta)$ and, so, $m_{\lambda}^\delta \geq 2(q+1)$.
\end{proof}

As an immediate consequence of Remark \ref{Rmk Codimensions and multiplicity} and Theorem \ref{ThmIdCnOrdinaryUT2} we have the following.
\begin{Lem}
\label{LemMolt3Righe}
Let $p\geq 1$ and $q\geq 0$. If $\lambda=(p+q,p,1)$, then in \eqref{CocharacterDer} we have $m_{\lambda}^\delta\geq q+1$.
\end{Lem}

We are now in a position to prove the following theorem about the $L$-algebra $UT_2^\delta$.

\begin{thm}
	\begin{enumerate}
		\item $\Id^{L}(UT_2^\delta)=\langle [x,y][z,w],[x,y]^{\delta},x^\delta [y,z], x^{\delta}y^{\delta}, x^{\delta^{2}}\rangle_{T_{L}}$.
		\vspace{1mm}
		
		\item $c_n^{L}(UT_2^\delta)=2^{n-1}n+1$.

\vspace{1mm}
\item If $\chi_{n}^{L}(UT_{2}^\delta)=\sum_{\lambda\vdash n} m_{\lambda}^\delta \chi_{\lambda}$ is the $n$th differential cocharacter of $UT_{2}^{\delta}$, then
$$m_{\lambda}^\delta =\begin{cases}
		n+1, & \mbox{ if } \lambda=(n) \\
		2(q+1), & \mbox{ if } \lambda=(p+q,p) \\
		q+1, & \mbox{ if } \lambda=(p+q,p,1) \\
		0 & \mbox{ in all other cases}
		\end{cases}.$$
	\end{enumerate}
\end{thm}
\begin{proof}
By Lemma \ref{Lemmma generatori UT^d} the $T_L$-ideal of differential identities of $UT_2^\delta$ is generated by the polynomials $[x,y][z,w],[x,y]^{\delta},x^\delta [y,z], x^{\delta}y^{\delta}, x^{\delta^{2}}$ and the elements in \eqref{EqBaseNoIdentityAllDerivations} are a basis of $P_{n}^{L}$ modulo $P_{n}^{L}\cap \Id^{L}(UT_{2}^\delta)$. Thus by counting these elements we get that $c_n^{L}(UT_2^\delta)=2^{n-1}n+1$.

Finally, as a consequence of Lemmas \ref{LemMoltRiga delta}, \ref{LemMolt2Righe delta}, \ref{LemMolt3Righe} and by following verbatim the proof of \cite[Theorem 12]{GiambrunoRizzo2018} we get the decomposition into irreducible characters of $\chi_{n}^{L}(UT_{2}^\delta)$.
\end{proof}


Notice that $\var^L(UT_2^\delta)$ has exponential growth, nevertheless it has no almost polynomial growth. In fact, the algebra $UT_2$ (ordinary case) is an algebra with $F\delta$-action where $\delta$ acts trivially on $UT_2$, i.e., $x^{\delta}\equiv 0$ is differential identity of $UT_2$. Then it follows that $UT_2\in \var^L(UT_2^\delta)$, but $\var^L(UT_2)$ growths exponentially. Thus we have the following result.

\begin{thm}
\label{Thm: $UT_2^delta$}
$\var^L(UT_2^\delta)$ has no almost polynomial growth.
\end{thm}

\

Since any derivation of $ UT_2 $ is inner (see \cite{CoelhoPolcinoMilies1993}), it can be easily checked that the algebra $\Der(UT_2)$ of all derivations of $UT_2$ is the 2-dimensional metabelian Lie algebra with basis $ \{\varepsilon,\delta\} $. Thus let suppose that $L$ is a $2$-dimensional metabelian Lie algebra and let denote by $ UT_2^D$ the $L$-algebra $UT_2$ where $L$ acts on it as the Lie algebra $\Der(UT_2)$. Giambruno and Rizzo  in \cite{GiambrunoRizzo2018} proved the following result.

\begin{thm}{\cite[Theorems 19 and 25]{GiambrunoRizzo2018}}
\label{ThmIdCnAllDerivations}
\begin{enumerate}
\item $\Id^{L}(UT_{2}^D)=\langle  [x,y]^{\varepsilon}-[x,y], x^{\varepsilon}y^{\varepsilon}, x^{\varepsilon^{2}}-x^{\varepsilon}, x^{\delta\varepsilon }, x^{\varepsilon\delta}- x^{\delta}\rangle_{T_{L}}$.

\vspace{1mm}
\item $c_{n}^{L}(UT_{2}^D)=2^{n-1}(n+2).$

\vspace{1mm}
		\item If $\chi_{n}^{L}(UT_{2}^D)=\sum_{\lambda\vdash n} m_{\lambda}^D \chi_{\lambda}$ is the $n$th differential cocharacter of $UT_{2}^{D}$, then
		$$m_{\lambda}^{D}=\begin{cases}
		2n+1, & \mbox{ if } \lambda=(n) \\
		3(q+1), & \mbox{ if } \lambda=(p+q,p) \\
		q+1, & \mbox{ if } \lambda=(p+q,p,1) \\
		0 & \mbox{ in all other cases}
		\end{cases}.$$
\end{enumerate}
\end{thm}

Since $x^\delta\equiv 0$ is a differential identity of $UT_2 ^\varepsilon$, $\var^L(UT_2 ^{\varepsilon}) \subseteq \var^L (UT_2^D)$. Then by Theorem \ref{UT2e}, we have the following.

\begin{thm}{\cite[Theorem 26]{GiambrunoRizzo2018}}
\label{Thm: $UT_2^D$}
$\var^L(UT_2^D)$ has no almost polynomial growth.
\end{thm}


\section{On differential identities of the Grassmann algebra}
\label{sec:G}

Let $L$ be a finite dimensional abelian Lie algebra and $G$ the infinite dimensional Grassmann algebra over $F$. Recall that $G$ is the algebra generated by 1 and a countable set of elements $e_{1},e_{2},\dots$ subjected to the condition $e_{i}e_{j}=-e_{j}e_{i}$, for all $i,j\geq 1$.

Notice that $G$ can be decomposed in a natural way as the direct sum of the subspaces
$$G_{0}=\Span_{F} \{e_{i_{1}}\dots e_{i_{2k}} \ \vert \ i_{1}<\dots < i_{2k},k\geq 0\}$$ 
and 
$$G_{1}=\Span_{F} \{e_{i_{1}}\dots e_{i_{2k+1}} \ \vert \ i_{1}< \dots  < i_{2k+1},k\geq 0\},$$
i.e., $G=G_{0}\oplus G_{1}$.

Let now consider the algebra $G$ where $L$ acts trivially on it. Since $x^\gamma \equiv 0$, for all $\gamma \in L$, is a differential identity of $G$, we are dealing with ordinary identities. Thus by \cite{KrakowskiRegev1973} we have the following results.

\begin{thm}
\label{ThmG}
\begin{enumerate}
\item $\Id^L(G)=\langle [x,y,z]\rangle_{T}$.
\vspace{1mm}
\item $c^L _{n}(G)=2^{n-1}.$
\vspace{1mm}

\item $\chi_{n}^L(G)=\sum_{j=1}^{n} \chi_{(j, 1^{n-j})}$.
\end{enumerate}
\end{thm}

\begin{thm}
$\var^L(G)$ has almost polynomial growth.
\end{thm}

Recall that if $g=e_{i_{1}}\dots e_{i_{n}}\in G$, the set $\Supp\{g\}=\{e_{i_{1}},\dots,e_{i_{n}}\}$ is called the support of $g$. Let now  $g_{1},\dots,g_{t}\in G_{1}$ be such that $\Supp\{g_{i}\}\cap \Supp\{g_{j}\}=\emptyset$, for all $i,j\in\{1,\dots,t\}$. We set
$$\delta_{i}=2^{-1}\ad g_{i},\quad  i=1,\dots,t.$$
Then for all $g\in G$ we have
$$\delta_{i}(g)=\begin{cases} 
0, &\mbox{ if } g\in G_{0} \\
g_{i}g , &\mbox{ if } g\in G_{1}
\end{cases}, \quad i=1,\dots,t.$$
Since for all $g\in G$, $[\delta_{i},\delta_{j}](g)=0$, $i,j\in\{1,\dots,t\}$, $L=\Span_{F}\{\delta_{1},\dots,\delta_{t}\}$ is a $t$-dimensional abelian Lie algebra of inner derivations of $G$. We shall denote by $\widetilde{G}$ the algebra $G$ with this $L$-action.

Recall that for a real number $x$ we denote by $\lfloor x \rfloor$ its integer part.

\begin{thm}{\cite[Theorems 3 and 9]{Rizzo2018}}
\label{ThmIdCn}
\begin{enumerate}
\item $\Id^{L}(\widetilde{G})=\langle [x,y,z], [x^{\delta_i}, y], x^{\delta_i \delta_j} \rangle_{T_{L}}$, $i,j=1,\dots,t$.
\vspace{1mm}

\item $c_{n}^{L}(\widetilde{G})=2^{t}2^{n-1}-\sum_{j=1}^{\lfloor t/2\rfloor}\sum_{i=2j}^{t}\binom{t}{i}\binom{n}{i-2j}.$
\vspace{1mm}

\item If $\chi_{n}^{L}(\widetilde{G})=\sum_{\lambda\vdash n} m_{\lambda}^{L}\chi_{\lambda}$ is the $n$th differential cocharacter of $\widetilde{G}$, then
$$m_{\lambda}^L=\begin{cases}
		\sum_{i=0}^{r}\binom{t}{i}, & \mbox{ if } \lambda=(n-r+1,1^{r-1}) \mbox{ and }  r<t \\
		2^{t}, & \mbox{ if } \lambda=(n-r+1,1^{r-1}) \mbox{ and }  r\geq t\\
		0 & \mbox{ in all other cases}
		\end{cases}.$$
\end{enumerate}
\end{thm}

Recall that two functions $\varphi_{1}(n)$ and $\varphi_{2}(n)$ are asymptotically equal and we write $\varphi_1(n)\approx\varphi_2(n)$ if $\lim_{n\to\infty}\varphi_1(n)/ \varphi_2(n)=1$. Then the following corollary is an obvious consequence of the previous theorem.
\begin{cor}
\label{Cor approximation c^L_n}
$c_{n}^{L}(\widetilde{G})\approx 2^{t}2^{n-1}$.
\end{cor}

Notice that by Corollary \ref{Cor approximation c^L_n} $\var^L(\widetilde{G})$ has exponential growth, nevertheless it has no almost polynomial growth. In fact, the Grassmann algebra $G$ (ordinary case) is an algebra with $L$-action where $\delta_{i}$, $i=1,\dots,t$, acts trivially on $G$, i.e., $x^{\delta_{i}}\equiv 0$, $i=1,\dots,t$, are differential identities of $G$. Then it follows that $G\in \var^L(\widetilde{G})$, but by Theorem \ref{ThmG} $c_n^L(G)=2^{n-1}$. Thus we have the following result.

\begin{thm}{\cite[Theorem 6]{Rizzo2018}}
$\var^{L}(\widetilde{G})$ has no almost polynomial growth.
\end{thm}

%


\begin{thebibliography}{20}
\bibitem{BenantiGiambrunoSviridova2004} F. Benanti, A. Giambruno, I. Sviridova, {\em Asymptotics for the multiplicities in the cocharacters of some PI-algebras}, Proc. Amer. Math. Soc. {\bf 132} (2004), no. 3, 669--679.

\bibitem{CoelhoPolcinoMilies1993} S.P. Coelho, C. Polcino Milies, {\em Derivations of upper triangular matrix rings}, Linear Algebra Appl. {\bf 187} (1993), 263--267.

\bibitem{GiambrunoLaMattina2005} A. Giambruno, D. La Mattina, {\em PI-algebras with slow codimension growth}, J. Algebra 284 (1) (2005) 371–-391.

%

\bibitem{GiambrunoRizzo2018} A. Giambruno, C. Rizzo, {\em Differential identities, $2\times 2$ upper triangular matrices and varieties of almost polynomial growth}, J. Pure Appl. Algebra {\bf 223} (2019), no. 4, 1710--1727.


\bibitem{GiambrunoZaicev2005book} A. Giambruno and M. Zaicev, Polynomial identities and asymptotic methods, Math. Surv. Monogr., AMS, Providence, RI, {\bf 122} (2005).

\bibitem{Gordienko2013JA} A.S. Gordienko, {\em Asymptotics of H-identities for associative algebras with an H-invariant radical}, J. Algebra {\bf 393} (2013), 92--101.

\bibitem{GordienkoKochetov2014} A.S. Gordienko, M.V. Kochetov, {\em Derivations, gradings, actions of algebraic groups, and codimension growth of polynomial identities}, Algebr. Represent. Th. {\bf 17} (2014), no. 2, 539--563.



\bibitem{Kemer1979} A.R. Kemer, {\em Varieties of finite rank}, Proc. 15-th All the Union Algebraic Conf., Krasnoyarsk. {\bf 2} (1979), p. 73 (in Russian).

\bibitem{Kharchenko1979} V.K. Kharchenko, {\em Differential identities of semiprime rings}, Algebra Logic {\bf 18} (1979), 86--119.

\bibitem{KrakowskiRegev1973} D. Krakowski, A. Regev, {\em The polynomial identities of the Grassmann algebra}, Trans. Amer. Math. Soc. {\bf 181} (1973), 429--438.

\bibitem{Malcev1971} J.N. Malcev, {\em A basis for the identities of the algebra of upper triangular matrices}, Algebra i Logika {\bf 10} (1971), 393--400.


\bibitem{Procesi2006} C. Procesi, Lie Groups: An Approach through Invariants and Representations, Universitext, Springer-Verlag, New York, 2006.


\bibitem{Rizzo2018} C. Rizzo, {\em The Grassmann algebra and its differential identities}, Algebr. Represent. Theory {\bf 23} (2020), no. 1, 125--134. 
\end{thebibliography}
\end{document}